\documentclass[11pt, reqno]{amsart}

\usepackage[utf8]{inputenc}
\usepackage[T2A]{fontenc}
\usepackage[ukrainian,russian,english]{babel}

\usepackage[english]{d-omega3}

\usepackage{amssymb}
\usepackage{graphicx}
\usepackage{xcolor}
\usepackage[all]{xy}

\usepackage{amsfonts,amsmath,amsthm,amscd,amssymb,latexsym}

\usepackage[sans]{dsfont}
\usepackage{xspace}
\usepackage[mathscr]{eucal}%
\usepackage{ifthen}

\usepackage[noadjust]{cite}

\def\dcl{\DeclareMathOperator}
\dcl{\Aut}{Aut}
\dcl{\Hot}{Hot}
\dcl{\supp}{supp}
\dcl{\ind}{ind}
\dcl{\End}{End}
\dcl{\Hom}{Hom}
\dcl{\Mor}{Mor}
\dcl{\walks}{Walks}
\dcl{\rwalks}{RWalks}
\dcl{\paths}{Paths}
\dcl{\sk}{sk}
\dcl{\cycles}{Cycles}
\dcl{\rcycles}{RCycles}

\dcl{\WPS}{WP}
\dcl{\GR}{GR}
\dcl{\im}{Im}

\def\k{{\Bbbk}}
\def\bp{bimodule problem\xspace}
\def\WP{WP\xspace}
\def\Ob{\mathop{\rm Ob}\nolimits}
\def\add{\mathop{\rm add}\nolimits}
\def\rep{\mathop{\rm rep}\nolimits}

\def\red{\mathop{\rm red}\nolimits}
\def\Ann{\mathop{\rm Ann}\nolimits}
\def\Rad{\mathop{\rm Rad}\nolimits}
\def\con{\mathop{\rm con}\nolimits}
\def\dim{\mathop{\rm dim}\nolimits}

\def\ds{\displaystyle}
\def\geqs{\geqslant}
\def\leqs{\leqslant}

\def\vi{\varphi}

\newcommand\bsK{\mathsf K}

\newcommand\bsV{\mathsf V}

\newcommand{\py}{{\mathsf p}}
\newcommand{\roots}[1]{\Re_{#1}^{+}}

\newcommand{\K}{\mathsf K}
\newcommand{\V}{\mathsf V}

\def\cC{\complement}
\def\define#1{\emph{#1}}

\def\dsum{\mathop\oplus\limits}

\def\omto{\mathop{\mapsto}\limits}
\def\omtol{\mathop{\longmapsto}\limits}

\makeatletter

\newenvironment{orditem}%
{\begin{list}{\vrule height .9ex width .6ex depth -.3ex }{%
      \parsep=\parskip \itemsep=0pt \partopsep=0pt \topsep=0pt plus 1pt
      \rightmargin=0pt \leftmargin=1.75\parindent
      \def\makelabel##1{\hss\llap{##1}}
      \listparindent=\parindent
      \itemindent=0pt
      \labelsep=.5\parindent
      \@beginparpenalty=10000
      }}
{\end{list}}

\newcounter{ordenumbri} \def\theordenumbri{\arabic{ordenumbri})}
\newenvironment{ordenumbr}%
{\begin{list}{\theordenumbri}{%
      \usecounter{ordenumbri}
      \parsep=\parskip \itemsep=0pt \partopsep=0pt \topsep=0pt plus 3pt
      \rightmargin=0pt \leftmargin=0pt
      \def\makelabel##1{##1\hskip.5em}\listparindent=\parindent
      \labelwidth=0pt
      \itemindent=\parindent
      \labelsep=0pt
      \@beginparpenalty=10000
      }}
{\end{list}}

\newcounter{ordenumi} \def\theordenumi{\arabic{ordenumi}.}
{\begin{list}{\theordenumi}{%
      \usecounter{ordenumi}
      \parsep=\parskip \itemsep=0pt \partopsep=0pt \topsep=0pt plus 3pt
      \rightmargin=0pt \leftmargin=0pt
      \def\makelabel##1{##1\hskip.5em}\listparindent=\parindent
      \labelwidth=0pt
      \itemindent=\parindent
      \labelsep=0pt
      \@beginparpenalty=10000
      }}
{\end{list}}

\newcounter{listenumli} \def\thelistenumli{(\alph{listenumli})}
{\begin{list}{\thelistenumli}{%
      \usecounter{listenumli}
      \parsep=\parskip \itemsep=0pt \partopsep=0pt
      \topsep=1pt plus1pt minus.5pt
      \rightmargin=0pt \leftmargin=3\parindent
      \def\makelabel##1{\hss\llap{##1}}\listparindent=\parindent
      \labelwidth=0pt
      \labelsep=.25em
      \@beginparpenalty=10000
      }}
{\end{list}}

\makeatother

\def\myto{{\,\xrightarrow{\hspace{10pt}}\,}}
\def\wt{\widetilde}

\newcommand{\bSigma}{\mathfrak{L}}
\newcommand{\bOmega}{\mathfrak{M}}

\def\si{{\Sigma}}
\def\Si{{{\Sigma}\,}}
\def\siz{{{\Sigma}_0}}
\def\sio{{{\Sigma}_1}}

\def\sizo{{{\Sigma}_1^0}}
\def\sioo{{{\Sigma}_1^1}}
\def\tisi{{\widetilde{\Sigma}}}

\def\hS{{\widehat {\Sigma}}}

\def\bsi{{\bSigma}}

\def\bsiz{{\bSigma}_0}
\def\bsio{{\bSigma}_1}
\def\bsit{{\bSigma}_2}

\def\btisi{{\wt{\bSigma}}}

\def\sT{{\mathit T}}
\def\rT{{\mathrm T}}
\def\rQ{{\mathrm Q}}
\def\rP{{\mathrm P}}
\def\bZ{{\mathbb Z}}

\def\ti{\Tilde}
\def\ca{\mathscr A}
\def\cb{\mathscr B}
\def\cc{\mathscr C}
\def\tica{\ti\A}

\newcommand{\sC}{\triangle}
\def\sg{G}

\def\tial{{\ti A}}
\def\tibe{{\ti B}}
\def\tiga{{\ti C}}
\def\tix{{\ti x}}
\def\tia{{\ti a}}
\def\Chi{\mathcal X}

\newcommand{\ww}{{\vspace{0.01ex}} }

\newtheorem{fact}{Fact}%[section]

\newtheorem{step}{Step}%[section]
\newtheorem{substep}{Substep}[step]
\newtheorem{subsubstep}{Subsubstep}[substep]

\def\best{\begin{step}\ww}
\def\enst{\ww\end{step}}
\def\besst{\begin{substep}\ww}
\def\ensst{\ww\end{substep}}
\def\bessst{\begin{subsubstep}\ww}
\def\enssst{\ww\end{subsubstep}}
\def\befa{\begin{fact}\ww}
\def\enfa{\ww\end{fact}}

\def\A{\mathcal A}

\let\ca=\A

\begin{document}

\title{Galois Coverings of One-Sided Bimodule Problems}%Paper title
\author{Vyacheslav Babych}%Authors names
\email{vyacheslav.babych@univ.kiev.ua}
\address{Department of Mechanics and Mathematics,
Taras Shevchenko National University of Kyiv, 64, Volodymyrs'ka St., Kyiv, Ukraine}
\author{Nataliya Golovashchuk}%Authors names
\email{golova@univ.kiev.ua}
\address{Department of Mechanics and Mathematics,
Taras Shevchenko National University of Kyiv, 64, Volodymyrs'ka St., Kyiv, Ukraine}

\abstract{english}{Applying geometric methods of $2$-dimensional cell complex theory, we construct a Galois covering of a bimodule problem satisfying some structure,
triangularity and finiteness conditions in order to describe the objects of finite representation type.
% For a considered class of bimodule problems, this is possible due to existence of quasi multiplicative basis and determination of bimodule problem properties by its covering.
Each admitted bimodule problem $\ca$ is endowed with a quasi multiplicative basis.
The main result shows that for a problem from the considered class having some finiteness restrictions and the schurian universal covering $\ti\ca$,
either $\ca$ is schurian, or its basic bigraph contains a dotted loop, or it has a standard minimal non-schurian bimodule subproblem.}

% \abstract{ukrainian}{We construct a Galois covering of a bimodule problem satisfying some structure,
% triangularity and finiteness conditions in order to describe the objects of finite representation type.
% % For a considered class of bimodule problems, this is possible due to existence of quasi multiplicative basis and determination of bimodule problem properties by its covering.
% Each admitted bimodule problems $\ca$ is endowed with a quasi multiplicative basis.
% The main result shows that for a problem from the considered class having some finiteness restrictions and the schurian universal covering $\ti\ca$,
% either $\ca$ is schurian, or its basic bigraph contains a dotted loop, or it has a standard minimal non-schurian bimodule subproblem.}
% 
% \abstract{russian}{We construct a Galois covering of a bimodule problem satisfying some structure,
% triangularity and finiteness conditions in order to describe the objects of finite representation type.
% % For a considered class of bimodule problems, this is possible due to existence of quasi multiplicative basis and determination of bimodule problem properties by its covering.
% Each admitted bimodule problems $\ca$ is endowed with a quasi multiplicative basis.
% The main result shows that for a problem from the considered class having some finiteness restrictions and the schurian universal covering $\ti\ca$,
% either $\ca$ is schurian, or its basic bigraph contains a dotted loop, or it has a standard minimal non-schurian bimodule subproblem.}


\shortAuthorsList{V.\,Babych, N.\,Golovashchuk}

% \udc{УДК 517.6544}


\maketitle

% 16P60   	Chain conditions on annihilators and summands: Goldie-type conditions [See also 16U20], Krull dimension
% 16P70   	Chain conditions on other classes of submodules, ideals, subrings, etc.; coherence
% 16D90   	Module categories [See also 16Gxx, 16S90]; module theory in a category-theoretic context; Morita equivalence and duality
% 16G20   	Representations of quivers and partially ordered sets
% 15A04   	Linear transformations, semilinear transformations
% 15A63   	Quadratic and bilinear forms, inner products
% 57M20   	Two-dimensional complexes
% 57M10   	Covering spaces

\section*{Introduction}

Description of bimodule problems of finite representation type is an important task of representation theory \cite{Dr2,GaRo,essen}.
A useful tool for a finiteness problem solution is so called ``covering method'' (\cite{Bib-BongartzCritical,DrOvs})
which is especially effective when the basis of bimodule problem is multiplicative, i.\,e. the composition of two composable basic elements is either zero or a basic element too~(\cite{RS}).
For a class of admitted bimodule problems, we introduced a quasi multiplicative basis (\cite{bgoadm2011}) generalizing the notion of a multiplicative one. Following \cite{BGOR}, we introduce standard minimal non-schurian admitted bimodule problem
and use the result of \cite{bg2} stating that minimal admitted non-schurian bimodule problem
with weakly positive Tits quadratic form is standard. A similar result was obtained by the authors for another class of bimodule problems previously in~\cite{bg0}.

%The existence of a quasi multiplicative basis give an explanation to the notion of
%a standard minimal non-schurian bimodule problem from the point of view of covering theory.

From the representation theory point of view,
it is important to determine the minimal non-schurian subproblems effectively,
and describe the classes of problems having a correspondence between the Tits form and the category of representations.
In the simplest cases, the dimensions of indecomposable representations correspond to the roots of the Tits form (\cite{Ga1,KR}).

We apply geometrical technique for investigation of bimodule problem representation category properties for a class of bimodule problems endowed with a quasi-multiplicative basis. The proofs use geometrical methods similar to those used in the geometric group theory \cite{Ols1}. We associate a $2$-dimensional cell complex with a faithful admitted bimodule problem in order to construct a corresponding Galois covering. This is the main tool in the geometrical investigation of bimodule problems. We use an universal covering technique to study the representation type of the initial admitted bimodule problem. A Galois covering of a bimodule problem of finite type induces a covering of corresponding representation categories with the same fundamental group (\cite{DOF,DrOvs}).

The main result (Theorem \ref{theorem-support-indecomposable-and-covering}) states the existence of standard minimal non-schurian subproblem for a finite dimensional non-schurian bimodule problem with weakly positive Tits form and schurian universal covering. After \cite{bgoadm2011,bg2}, this is the next step on the way of representation type characterization for finite dimensional problems from the considered class.

We use the definitions, notations and statements from \cite{bg,bgoadm2011,essen,bg2}. The considered class $\cC$ of bimodule problems and the notion of quasi multiplicative basis are defined in \cite{bgoadm2011}. The notions and facts from the theory of quadratic forms can be found in \cite{Dr2,KR,Ri}.

\section{Preliminaries}

%\section{Bigraphs and walks}

\subsection{Bigraphs.}
\label{bigraph}

A \define{directed bigraph} $\Sigma=(\Sigma_0,\Sigma_1,s,e,\deg)$ is given by a set $\Sigma_0$ of vertices, a set $\Sigma_1$ of arrows,
two maps $s,e:\Sigma_1\to\Sigma_0$ defining an initial and a terminal vertices of an arrow,
and the degree $\deg:\Sigma_1\to\{0,1\}$ indicating a type of an arrow.
For $X$, $Y\in\siz$, let $\sio(X,Y)=\{x\in\sio\mid s(x)=X$, $e(x)=Y\}$.
Here $\Sigma_1=\Sigma_1^0\sqcup\Sigma_1^1$, where $\Sigma_1^i=\deg^{-1}(i)$, $i=0,1$.
The arrows from $\sizo$ ($\sioo$) are called \define{solid} (\define{dotted}).
If $\Sigma^\prime=(\Sigma_0^\prime,\Sigma_1^\prime,s',e',\deg')$ is another  bigraph,
then a morphism $f:\Sigma\to\Sigma^\prime$ of  bigraphs is a pair
$f=(f_0,f_1)$, where $f_0:\Sigma_0\to\Sigma_0^\prime$, $f_1:\Sigma_1\to\Sigma_1^\prime$
such that  $f_0 s= s' f_1$, $f_0 e=e' f_1$.
A bigraph $\Sigma$ is called \define{connected} (\define{$0$-connected}) if there does not exist a decomposition $\Sigma_0=\Sigma_0^1\sqcup\Sigma_0^{2}$ into nonempty
subsets such that $s(a)\in\Sigma_0^i$
implies $e(a)\in\Sigma_0^i$ for every $a\in\sio$ ($a\in\Sigma_1^0$), $i=1,2$. A connected bigraph is called:
\begin{orditem}
\item
a \define{loop}, if $|\Sigma_0|=|\Sigma_1|=1$;
\item
a \define{circle} $C_{n}$ if $|\siz|=|\Sigma_1|=n\geqs 2$,
and $|e^{-1}(A)|+|s^{-1}(A)|=2$ for any $A\in\Sigma_0$;
\item
a \define{chain}, if there are $X_1,X_2\in\Sigma_0$
such that  $|e^{-1}(X_i)|+|s^{-1}(X_i)|=1$, $i=1,2$,
and $|e^{-1}(A)|+| s^{-1}(A)|=2$ for every $A\in\Sigma_0\backslash\{X_1,X_2\}$.
\end{orditem}

A bigraph $\Sigma$ is called \define{locally finite} if any $X\in\Sigma_{0}$ is incident to finitely
many arrows, and \define{finite} provided $\Sigma_0$ and $\Sigma_1$ are finite.

\subsection{Paths and walks in bigraph.}\label{intr06}

For a bigraph $\Sigma=(\Sigma_0,\Sigma_1,s,e,\deg)$ let $\Sigma_1^{-1}$
be the set of the elements $x^{-1}$ for all $x \in\Sigma_1$,
and let ${\hS}$ be the bigraph such that
$\hS_0=\Sigma_0, \hS_1=\Sigma_1\sqcup\Sigma_1^{-1}.$
The maps $s,e:\hS_1\myto\hS_0$ and $\deg$ restricted on $\Sigma_1\subset\hS_1$
coincide with those for $\Sigma$, and
$e(x^{-1})=s(x)$, $s(x^{-1})=e(x)$, $\deg x^{-1}=\deg x$.
%On  $\hS_1$ is defined the involution ${}^{-1}$, $ (x^{-1})^{-1} =  x,  (x)^{-1} =  x^{-1}, x\in \Sigma_1$.

A \define{walk} (a \define{path}) on $\Sigma$ of the \define{length} $m$ is a sequence $\omega=x_1x_2\ldots x_m$
of arrows $x_i \in{\hS}_1$ ($x_i\in\Sigma_1$), $1\leqslant  i\leqslant m$, such that  $s(x_i)=e(x_{i+1})$ for all $1 \leqslant i < m$.
%In the pictures an arrow $x^{-1}\in{\hS}_1$ we will display  as
%corresponding arrow $x\in\Sigma$, passed in direction from $e(x)$ to $s(x)$.
%
%
%\begin{figure}[htbp]
%\begin{center}
%\includegraphics[width=8cm]{dummy}
%\caption{Examples of walks and pathes}
%\end{center}
%\end{figure}
%
% The arrows $x_1,\ldots,x_m$ and their begins and ends are said \define{to be incident} to the walk $\omega=x_1\ldots x_m$.
% A walk $\omega=x_1\ldots x_m$ with $e(x_1)=s(x_m)$ is called a \define{cyclic walk}.
We assume that $s(\omega)=s(x_m)$, $e(\omega)=e(x_1)$, and for every $X\in\Sigma_{0}$ there exists a unique (\define{trivial}) path ${\mathds 1}_{X}$ of the length $0$
such that $s({\mathds 1}_{X})=e({\mathds 1}_{X})=X$.
%A walk $\omega$ is called \define{solid} (\define{dotted}) provided that it contains only solid (only dotted) arrows.
%\begin{Agreement} { Writing down a cyclic walk
% $\omega=x_1\ldots x_m$
% of the length $m>0$ we use the cyclic
%modulo $m$ enumeration of the arrows and
%vertices, incident to $\omega$, f.g.
%$x_k=x_{im+k}$, especially $x_1=x_{n+1}$.}
%\end{Agreement}
The \define{composition} $\omega\omega'$ of the walks $\omega$ and $\omega'$
is naturally defined if $s(\omega)=e(\omega')$.
% The \define{decomposition} of the walk $\omega$ in two walks $\omega_1,\omega_2$ means $\omega=\omega_1\omega_2$.
% If $\omega=\omega_1\omega^\prime\omega_2$, then $\omega^\prime$ is called
% \define{a subwalk} of the walk $\omega$.

%If $\omega$ does not contain cyclic subwalk, then $\omega$ is a walk
%\define{without self-intersection} or \define{simple} walk.
%A cyclic walk $\omega$, containing unique cyclic subwalk $\omega$ we call \define{simple cycle}.

The walks and the paths on $\Sigma$ form the categories denoted by $\walks_{\Sigma}$ and $\paths_{\Sigma}$ respectively,
both with the set of objects $\Sigma_0$.
For any $X,Y\in\Sigma_0$ the set $\walks_{\Sigma}(X,Y)$ ($\paths_{\Sigma}(X,Y)$)
contains all the walks (the paths) $\omega$ in $\Sigma$ such that  $s(\omega)=X$, $e(\omega)=Y$.
There exists an obvious inclusion functor $\imath:\paths_{\Sigma}\hookrightarrow\walks_{\Sigma}$.
%Analogously are introduced the categories of all solid paths $\paths_{\Sigma}^0$ and walks $\walks_{\Sigma}^0$ on $\Sigma$.
We denote by ${\mathsf W}_{\Sigma}$ the \define{groupoid} of walks on $\Sigma$,
${\mathsf W}_{\Sigma}=\walks_{\Sigma}/(x\circ x^{-1}={\mathds 1}_{e(x)},x\in \hS_1)$,
and by $r_\Sigma:\walks_\Sigma\myto{\mathsf W}_{\Sigma}$ the canonical projection.

A morphism $f=(f_0,f_1):\si\to\si'$ of bigraphs naturally induces a functor $\walks_f:\walks_\si\to\walks_{\si'}$, namely $\Ob\walks_\si=\si_0\ni X\mapsto f_0(X)\in\Ob\walks_{\si'}=\si'_0$, $\walks_\si(X,Y)\ni\omega=x_1\ldots x_n\mapsto\omega'=f_1(x_1)\ldots f_1(x_n)\in\walks_{\si'}(f_0(X),f_0(Y))$ for any $X,Y\in\Ob\walks_\si$.

\subsection{Reduced walks.}
A walk $\omega$ is called \define{reduced} if it does not contain a walk $xx^{-1}$ for some $x\in\hS_1$,
and \define{reducible} in opposite case.
The operation $\omega_1 xx^{-1}\omega_2\mapsto\omega_1\omega_2$ we call a \define{reducing} of the walk (by the pair $x x^{-1}$).

If $\omega=\omega_1\omega_2$ is a cyclic walk, then $\omega'=\omega_2\omega_1$ is cyclic as well,
and we say that $\omega$ and $\omega'$ are \define{cyclically equivalent}.
A cyclic walk $\omega$ is called \define{cyclically reduced}
if every cyclically equivalent walk is reduced.
A class $\langle\omega\rangle$ of the cyclic equivalence containing the cyclic walk $\omega$ is called a \define{cycle}.
We denote by $\cycles_{\Sigma}$ the set of all cycles on $\Sigma$, and by $\rwalks_\Sigma$ ($\rcycles_\Sigma$)
the set of all reduced walks (cyclically reduced cycles) on $\Sigma$.

%\begin{remark}[-]
%\label{contractibility_cyclic_invariant}
%Every walk or cyclic walk can be uniquely
%reduced \cite{ZVC}, that means, there exist
%uniquely defined maps
%$\red:\walks_\si\myto\walks_\si$, $\red
%c:\cycles_\si\myto\cycles_\si$, such that
%
%\begin{itemize}
%
%\item  $\Im\red\subset\rwalks_\si$,
%$\Im\red c\subset\rcycles_\si;$
%
%\item  for every walk ${\sigma}$ holds
%$r_\Si(\red({\sigma}))=r_\Si({\sigma})$;
%
%\item  if $\langle\omega\rangle=\red c(<{\sigma}>)$
%and $\omega,{\sigma}:A\mytoA$,
%then $r_\Si(\omega)=r_\Si({\sigma})$.
%\end{itemize}
%\end{remark}
%%%%\footnote{In this paper we use the word "reduced" in two senses. The
%%%%first (geometrical) sense is presented above. }
%
%A morphism $f:\Si\myto \Si^\prime$ of
%bigraphs induces natural functors
%$\walks_f$, $\paths_f$,${\mathsf W}_f$ etc.

%\section{2 - dimensional complexes and
%homotopies}\label{two_complexes}

A bigraph morphism $f=(f_0,f_1):\si\to\si'$ induces in a natural way a map
$\cycles_{\si}\ni\langle\omega\rangle \omtol^{\cycles_f}\langle\walks_f(\omega)\rangle\in\cycles_{\si'}$.

\subsection{2-dimensional complex over a bigraph.}

An (abstract) $2$-dimensio\-nal \define{complex} (shortly, complex) $\bSigma$ is defined by:
\begin{orditem}
\item
a bigraph $\Sigma=(\Sigma_0,\Sigma_1,s,e,\deg)=\sk_1(\bsi)$;
the vertices from $\bsiz=\siz$ we call also \define{$0$-dimensional cells},
and the arrows from $\bsio=\sio$ we call \define{$1$-dimensional cells};
\item
a set $\bsit=\{\sC_i, i\in I=I(\bsi)\}$ which elements are called \define{$2$-dimen\-sional cells} or $2$-\define{cells};
\item
a \define{boundary map} $\partial=\partial_\bsi:\bsit\myto\rcycles_{\Sigma}$.
\end{orditem}

% We say, that the complex $\bsi$ is defined over the bigraph $\si=\sk_1(\bsi)$,
% in the case $\bsit=\varnothing$, we identify $\bsi$ with the bigraph $\sk_1(\bsi)$.
% A complex and corresponding bigraph we will denote by the same letter,
% but {the complex will be typeset bold}.
We say that $A\in\bsiz$ ($x\in{\widehat\bsi}_1=\widehat\si_1$) belongs or is incident to $\sC_i$
and write $A\in\sC_i$ ($x\in\sC_i$) provided $A$ (at least one of $x^{\pm 1}$)
is incident or belongs to some element of $\partial(\sC_i)$.

If $\bsi^\prime$ is a complex, then a \define{morphism} $\bsi\to\bsi'$ of complexes
is a triple $f=(f_0,f_1,f_2)$ such that $(f_0,f_1)$ is a morphism of the underlying bigraphs
(denoted by the same letter $f$),
and $f_2:\bsit\myto\bsi_2^\prime$ is a map satisfying $\cycles_f\partial_\bsi=\partial_{\bsi'}f_2$.
The morphism $f$ is called \define{monomorphism (epimorphism)}
provided all $f_{i}$ are monomorphisms (epimorphisms).

The notion of \define{subcomplex} of a  complex
is defined in a standard way.
If $S\subset\bsiz\sqcup\bsio\sqcup\bsit$, then by $[S]$
we denote the subcomplex in $\bsi$, generated by $S$,
i. e. the minimal subcomplex in $\bsi$ containing $S$.
For $S\subset \Sigma_{0}$ by $\Sigma_S$ and $\bSigma_S$ we denote the restriction of $\Sigma$ and $\bSigma$ to $S$.

Let $\bsi$ be a complex.
We say that the complex $\bOmega$ together with a morphism $\py:\bOmega\to\bsi$
forms a \define{complex over} $\bsi$.
The morphism $f:(\bOmega,\py)\to(\bOmega',\py')$ of complexes over $\bsi$
is a morphism of complexes $f:\bOmega\myto\bOmega'$ such that $\py=\py'f$.

%Let $\bsi$ be a complex. A \define{complex over $\bsi$} we call a complex $\bOmega$
%together with fixed morphism $\py:\bOmega\myto\bsi$ and morphism
%$f:(\bOmega,\py)\myto(\bOmega',\py')$ of complexes over $\bsi$ consist of morphism
%of complexes $f:\bOmega\myto\bOmega'$, such that $\py=\py'f$.

\subsection{Quotient complex.}

Let  $\bsi$ be a complex, let $\sim$ be an equivalence relation on $\bsiz\sqcup\bsio$ such that $A\not\sim a$ for any $A\in\si_0$, $a\in \si_1$, and if $a\sim b$ for some $a,b\in \si_1$, then $\deg_\si(a)= \deg_\si(b)$, $s(a)\sim$ $s(b)$, $e(a)\sim$ $e(b)$. A \define{quotient bigraph} $\si/_{\sim}$ is defined by the set $\bsiz/_{\sim}$ of vertices, the set $\bsio/_{\sim}$ of arrows, and correctly induced $e$, $s$ and $\deg$. Let  $\py=(\py_0,\py_1):\si\myto \si/_{\sim}$ be the natural bigraph epimorphism. If $\cycles_\py(\partial_{\bsi}(\sC))$ is cyclically reduced for any 2-cell $\sC\in\bsi_2$, then we set $(\bsi/_{\sim})_{2}=\bsit$ and $\partial_{\bsi/_{\sim}}=\cycles_\py\partial_{\bsi}$ that defines the quotient complex $\bsi/_{\sim}$ and the morphism $\py=(\py_0,\py_1,\py_2):\bsi\myto\bsi/_{\sim}$ where $\py_2=1_{\bsi_2}$.

\subsection{Homotopy relation.}

The structure of a $2$-dimensional complex $\bsi$ over a bigraph $\si$ induces the homotopy relation on $\walks_{\Sigma}$.
For any $\omega,\omega_1,\omega_2\in\walks_\si$,
a following transformations of walks:
\begin{ordenumbr}
\item  $\omega_1 x x^{-1} \omega_2\leadsto\omega_1 \omega_2$
or $\omega_1\omega_2\leadsto\omega_1 x x^{-1}\omega_2$
for any $x\in{\hS}_1$;
\item
$\omega_1\omega_2\leadsto\omega_1\omega\omega_2$
or $\omega_1\omega\omega_2\leadsto\omega_1\omega_2$
such that $ \langle\omega\rangle=\partial(\sC)$, $\sC\in\bsit$;
%$\omega,\omega_1$, $\omega_2\in\walks_\bsi$.
\end{ordenumbr}
are called \define{elementary homotopies}.

Two walks $\omega,\omega^{\prime}$ on $\si$ are said to be \define{homotopic}
provided there exists  a sequence $E= (E_1,\ldots,E_N)$ of elementary homotopies
such that $\omega=\omega_0\stackrel{E_1}{\leadsto}\omega_1\stackrel{E_2}{\leadsto}
\ldots\stackrel{E_N}{\leadsto}\omega_N=\omega^\prime$, $N\geqs 0.$
In this case, we  write
$\omega\stackrel{E}{ \leadsto}\omega^{\prime}$ or $\omega \sim\omega^{\prime}$ and say that $E$ is \define{a homotopy} between $\omega$ and $\omega^\prime$.

We indicate the simple properties of homotopies:
\begin{ordenumbr}
\item
the relation $\sim$ is an equivalence;
\item
if $\omega \sim \omega^{\prime}$, then $s(\omega) = s(\omega^{\prime})$, $e(\omega) = e(\omega^{\prime})$;
\item
if $\omega_1\sim\omega_1^\prime$, $\omega_2\sim\omega_2^\prime$,
then $\omega_1\omega_2\sim\omega_1^\prime\omega_2^\prime$
provided one of the compositions is defined.
\end{ordenumbr}

For a walk $\omega$ we denote by $[\omega]$ the \define{homotopic class of $\omega$}. Once the composition $\omega\omega'$ of walks $\omega$, $\omega'$ is specified, the composition of their classes is correctly defined by the equality
$[\omega]\cdot [ \omega^{\prime} ] = [ \omega\omega^{\prime} ]$.
We denote by $\Hot_{\bsi}$ the quotient category of homotopic classes of walks on $\si$,
$\Hot_{\bsi}(X,Y)=\walks_{\si}(X,Y)/_{\sim}$.
Note that $\Hot_{\bsi}$ is a groupoid.

%\subsection{Contractible walks and fundamental group of a complex}\label{intr010}

\subsection{Fundamental group of a complex.}\label{intr010}
A cyclic walk on $\bsi$ is called \define{contractible} if it is homotopic to a trivial walk.
% we write ${\sigma} \sim 0$.
%\begin{remark}
%If ${\sigma} \sim 0$ and ${\sigma}_1$ is \todo{cyclically equivalent} to ${\sigma}$, then $\sigma_1 \sim 0$.
%\end{remark}
For any $B\in\bsiz$, we define \define{the fundamental group} $\sg(\bsi,B)$ of $\bsi$
with the base vertex $B$  as $\Hot_\bsi(B,B)$.
If $\bsi$ is connected, then the fundamental groups with different base vertices are isomorphic,
that allows  to define \define{the fundamental group} $\sg(\bsi)$ of the connected complex $\bsi$.

% For $n\geqs 1$, a \define{polygon} (or $n$--gon) is a complex $\bsi$
% such that  $\sk_1(\bsi)$ is a circle or a loop $a_1\ldots a_n$, $|\bsiz|=|\bsio|=n$ and
% $\bsit=\{\sC_1\}$, $\partial(\sC_1)=\langle a_1\ldots a_n\rangle\in\rcycles_\bsi$.
% Obviously, $\partial(\sC_1)$ is contractible.
%\begin{figure}[htbp]
%\begin{center}
%\includegraphics[width=7.2cm]{n_gons.eps}
%\caption{Examples of polygons.}
%\end{center}
%\end{figure}

\subsection{Covering of complex.}
%%%Let $A\in\bsiz$ be a vertex in a complex $\bSigma$.
%%%\define{The star} of the vertex $A\bigstar(A)=\bigstar_{\Sigma}(A)$
%%%is called the minimal subcomplex, containing all the elements from $\bsio$ and $\bsit$,
%%%incident to $A$.
An epimorphism $\py:\btisi\to\bsi$ of complexes
is called \define{a covering morphism} (or \define{covering} of $\bsi$)
with the base \mbox{$\bsi$}, if:
\begin{orditem}
\item
for every $\omega\in\walks_\bsi(X,Y)$ and ${\wt X}\in\py_0^{-1}(X)$ (${\wt Y}\in\py_0^{-1}(Y)$),
there exist an unique ${\wt Y}\in\btisi_0$ (${\wt X}\in\btisi_0$)
and a unique walk ${\wt \omega}:{\wt X}\myto{\wt Y}$ such that $\walks_\py({\wt \omega})=\omega$
(the property of the uniqueness of lifting of walks, \cite{Wh});

\item
for every $\sC\in\bsit$, $A\in\bsiz$ such that  $A\in\partial(\sC)$ and ${\wt A}\in\py_0^{-1}(A)$
there exists unique ${\wt\sC}\in{\wt\bsi}_2$ such that  ${\wt A}\in\partial({\wt\sC})$
and $\py_2({\wt\sC})={\sC}$ (the property of homotopy lifting uniqueness, \cite{Wh}).
\end{orditem}

%\begin{Agreement}
% \noindent
% Unless otherwise stated, we consider the coverings with a connected $\tisi$.
%\end{Agreement}

Coverings with a fixed base $\bsi$ form the subcategory in the category of complexes over $\bsi$.
An object $\py:\btisi\myto\bsi$ in this category is called a \define{universal covering} of $\bsi$ if every morphism $f:\py'\to\py$ is an isomorphism.

\subsection{Construction of universal covering.}\label{intr012}

We construct the universal covering $ \btisi$ of the complex $\bsi$ similarly to \cite{DrOvs,bg}.
\begin{ordenumbr}
\item
We fix $B \in\bsiz$ and define the set ${\btisi}_0$ by
$\ds{\sqcup_{X\in\bsiz}\Hot_\bsi(B,X)}$.
\item
The arrows in $\btisi$ are the pairs $( [ \omega ], x)$ of $[ \omega ]\in\btisi_0$,
$x\in\bsio$ such that $s(x)=e(\omega)$, $\deg_\btisi(([\omega ],x))=\deg_\bsi(x)$,
$s(([\omega],x))=[\omega ]$, $e(([ \omega ],x))=[ x\omega ]$.
\item
The covering morphism
$\py : \ti \si\myto\si$ on the underlying bigraphs is defined by
$\py_0([\omega])=e(\omega)$, $\py_1(([\omega ],x))=x$.
Obviously, the constructed $\py$ has a property of the uniqueness
of lifting of walks.
\item
If $\sC\in\bsi_2$ and $\langle\omega\rangle=\partial_\bsi(\sC)$,
then every walk ${\ti\omega}\in\walks_\py^{-1}(\omega)$ is also cyclic
since $\omega$ is contractible on $\bsi$. Let $\Theta=\ds{\sqcup_{\sC\in\bsit}}\cycles_\py^{-1}(\partial_\bsi(\sC))$.
Then we define $\btisi_2$ by the set $\{{\ti \sC_{\theta}}\}_{{\theta} \in \Theta}$,
the boundary map in $\btisi$ by the equality ${\partial}_\bsi({\ti\sC_{\theta}})=\theta$,
and ${\py}_2({\wt {\sC}}_{\theta})=\sC$
provided $\cycles_\py({\theta})=\partial_\bsi(\sC)$.
\end{ordenumbr}

The fundamental group $\sg=\sg(\bsi,B)$ acts on $ \btisi$ as follows:
if
$[\omega_B]\in\sg(\bsi,B)$, $[\omega]\in\btisi_0$,
then $[\omega][\omega_B]=[\omega\omega_B]\in\btisi_0$.
The structure maps of $\btisi$ commutes with the action of $\sg$,
hence a quotient complex
$\btisi/\sg=(\sg\btisi_0,\sg \btisi_1,\sg \btisi_2)$ is defined,
and the morphism $\py$ induces an isomorphism
of the complexes $\btisi/\sg$ and $\bsi$.

\subsection{Tits quadratic form of bimodule problem.}

We refer to~\cite{bgoadm2011} for the notations and facts from the theory of bimodule problems. A \define{quadratic form} $q =q_{\si}$ of the bigraph $\si$  is defined on $x=(x_A)_{A\in\si_0}\in\bZ^{\si_0}$ by
$$\ds q(x)={\sum}_{(A,B)\in\si_0\times\si_0}(\delta_{AB}+ |\si_1^1(A,B)|-|\si_1^0(A,B)|)x_Ax_B$$
where $\delta$ is the Kronecker delta.
By definition, the \define{Tits quadratic form $q_\ca$ of bimodule problem $\ca$} is a quadratic form $q_{\si_\ca}$ of a basic bigraph $\si_\ca$.

% Recall some facts from the theory of integer quadratic forms (\cite{bg2}).
We denote by $\roots{q}$ the set of all positive roots of $q$ (i. e. $x>0$ such that $q(x)=1$).
For a quadratic form $q$, by $(,)$ or $(,)_q$ we denote the
corresponding symmetric bilinear form, $(x,y)=(q(x+y)-q(x)-q(y))/2$.
If $q$ has a sincere positive root, then $q$ is called \define{sincere}.

The form $q$ is called \define{weakly positive} provided  $q(x)>0$ for $x>0$.
We denote by ${\WPS}$ the set of all weakly positive locally finite forms.

If $q(x_1,\dots,x_n)\in\WPS$, $n\geqs 2$, $e_i$ is a simple root and $z\in\roots{q}$ is sincere,
then $2(e_i,z)\in\{0,1,-1\}$
and $z-2(e_i,z)e_i\in\roots{q}$.
A sincere $x\in\roots{q}$ is called a \define{basic root} if there exist $i_1,i_2\in I=\{1,\dots,n\}$
such that  $2(e_{i_1},x)=2(e_{i_2},x)=1$, $x_{i_1}=x_{i_2}=1$ and $2(e_i,x)=0$, $i\in I\backslash\{i_1,i_2\}$.
We call $i_1$, $i_2$ the \define{singular} vertices of $x$.

% \begin{lemma}[\cite{bg2}]\label{Lemma-le_tpoints}
% Let $q(x_1,\dots,x_n)\in\WPS$ be an unit form, $n\geqs3$, let $x\in \roots{q}$ be sincere non-basic, and let $i_1,i_2\in I$, $i_1\ne i_2$.
% Then there exists non-sincere $y\in \roots{q}$ such that  $y<x$,
% and $i_1, i_2\in\supp y$.
% \end{lemma}

\subsection{Representation category.} For a bimodule problem $\ca=(\bsK,\bsV)$, \define{a representation} $M$ of $\ca$ is a pair $M=(M_\bsK,M_\bsV)$ where $M_\bsK\in\Ob\add\bsV=\Ob\add\bsK$ and $M_\bsV\in\add\bsV(M_\bsK,M_\bsK)$. If $M$, $N$ are two representations of $\ca$, then a \define{morphism} $f$ from  $M$ to $N$ is a morphism $f\in\add\bsK(M_\bsK,N_\bsK)$ such that $N_\bsV f-f M_\bsV=0$. The composition of morphisms and the unit morphisms in the \define{representation category} $\rep\ca$ and in the category $\add\bsK$ coincide. All indecomposable representations form the subcategory in $\rep{\mathcal A}$ which we denote by $\ind{\mathcal A}$.

% \subsection{Definition of covering morphism} All
% the \bp's supposed to be normal.

The category $\rep{\mathcal A}$ is a Krull-Schmidt category. Bimodule problem $\A$ is called
of \define{finite representation type} provided $\rep{\mathcal A}$
has finitely many isoclasses of indecomposable objects,
and of \define{infinite representation type} in opposite case. $\A$ is called
\define{locally representation-finite} provided for any object $A\in\Ob\bsK$, there are finitely many isoclasses of indecomposable representations $M$ such that $A\in\supp\dim_\A M$.

A representation $M\in \rep {\mathcal A}$ is called \define{sincere}
provided $(\dim M)_A\ne0$ for any
$A\in\Ob\bsK$. A bimodule problem ${\mathcal A}$ is called
\define{sincere} if there exists a sincere indecomposable
representation $M\in \rep {\mathcal A}$.

A representation $M\in\ind \ca$ is called \define{schurian} provided it has only scalar endomorphisms.
A bimodule problem ${\ca}$ is called \define{schurian} provided every $M\in\ind \ca$ is schurian (\cite{Ga,KR}).

\begin{lemma}[\cite{Ga1,KR}]\label{Lemma-le_defshur}
Let $\A$ be a finite dimensional schurian
\bp. Then $\A$ is  representation finite,
its Tits form $q_{\A}$ is unit integral  and \WP,
the map $\dim_\A:\ind \A/_{\simeq}\to\roots{q_{\A}}$ is a bijection, where $\ind \A/_{\simeq}$ denote the set of all isoclasses of indecomposable representations.
\end{lemma}

\subsection{Covering of bimodule problem.}
A bimodule problem $\ti{\ca}=(\ti \K , \ti \V )$ is called a \define{covering} of the bimodule problem ${\ca}=(\K ,\V )$
provided there exists a  morphism $\py =(\py_0,\py_1):\ti {\ca}\myto{\ca}$ (\define{covering morphism})
such that $\py_{0}$ is a surjection on the objects and $\py_0,\py_1$
induce the following isomorphisms for any $\ti{A},\ti{B}\in\Ob{\ti\K}$ (\cite{DrOvs}):
\begin{gather*}
\dsum_{\ti{X}\in\py_0^{-1}(\py_0(\ti B))\hskip-3em} \ti
\K (\ti{A},\ti{X}) \simeq
\K (\py_0(\ti{A}),\py_0(\ti{B})),\quad
\dsum_{\ti{Y}\in\py_0^{-1}(\py_0(\ti{A}))\hskip-3em} \ti
\K(\ti{Y},\ti{B})\simeq \K(\py_0(\ti{A}),\py_0(\ti{B})),\\
\dsum_{\ti{X}\in\py_0^{-1}(\py_0(\ti{B}))\hskip-3em} \ti\V (\ti{A},\ti{X}) \simeq \V(\py_0(\ti{A}),\py_0(\ti{B})),\quad
\dsum_{\ti{Y}\in\py_0^{-1}(\py_0(\ti{A}))\hskip-3em} \ti\V (\ti{Y},\ti{B})\simeq \V(\py_0(\ti{A}),\py_0(\ti{B})).
\end{gather*}
%\begin{gather*}
%\dsum_{X\in\py_0^{-1}(B)}\hskip-1em \ti
%\K (A,X) \simeq
%\K (\py_0(A),\py_0(B)),\quad
%\dsum_{Y\in\py_0^{-1}(A)}\hskip-1em \ti
%\K (Y,B)\simeq  \K
%(\py_0(A),\py_0(B)),\\
%\dsum_{X\in\py_0^{-1}(B)}\hskip-1em \ti
%\V (A,X) \simeq \V (\py_0(A),
%\py_0(B)),\quad
%\dsum_{Y\in\py_0^{-1}(A)}\hskip-1em \ti
%\V (Y,B)\simeq
%\V (\py_0(A),\py_0(B)).
%\end{gather*}
In this case the bimodule problem $\ca$ is called a \define{base} of the covering.

The coverings with a fixed base form a category over $\ca$ in a standard way:
a morphism of the coverings $\py:\ti{\ca}\myto{\ca}$ and
$\py^{\prime}:\ti{\ca^{\prime}}\myto{\ca}$
is a morphism of bimodule problems  $\rho:\ti {\ca}
\myto\ti {\ca^{\prime}}$ such that $\py=\py^{\prime}\rho$.

The functor $\py_{*}=\rep\py:\rep\ti{\ca}\myto\rep{\ca}$ between
the representation categories induced by the  morphism $\py: \ti {\ca} \myto {\ca}$
of bimodule problems is called the \define{push-down} functor.
It allows to compare the representation types of bimodule problems  $\ti \ca$ and $\ca$.

\subsection{Galois covering.}
\label{galois_covering}
Let $\ti {\ca}=(\ti \K , \ti \V )$ be a locally finite dimensional bimodule problem, let
$G$ be a group acting freely on $\ti\K$ and $\ti\V$ that means:
\begin{orditem}
\item
there is given a group monomorphism $\sT:G\myto\Aut_{{\k}}(\ti\K,\ti\K)$
%where $\Aut_{{\k}}(\ti\K,\ti\K)$ is the automorphism group of $\ti\K$,
which defines a free action $\sT_{\ti{\K}}:G\times\ti\K \myto\ti\K$ of $G$ on $\ti\K$ such that\\
$G\times\Ob{ \ti \K }\ni(g,A)\omtol^{\sT_{\ti{\K}}}gA=\sT(g)(A)\in\Ob{\ti\K }$;
%is a free (without fixed points) action;

\item
an action $\sT_{\ti{\V}}:G\times\ti\V \myto\ti\V$ of $G$ on ${\ti{\V}}$
is a family of ${\k}$-isomorphisms
%\newline\centerline{
$\sT_{\ti{\V}}(g)(A,B):{\ti{\V}}(A,B)\myto{\ti{\V}}(g A,g B)$
%}\newline
such that $$\sT_{\ti{\V}}(g)(f_1 v f_2)=\sT_{\ti{\K}}(g)(f_1)\sT_{\ti{\V}}(g)(v)\sT_{\ti{\K}}(g)(f_2)$$ for any $f_1,f_2\in\ti\K$, $v\in\ti\V$ and $g\in G$ once the composition $f_1 v f_2$ is specified.
\end{orditem}

For a locally finite dimensional bimodule problem $\ti {\ca}=(\ti \K , \ti \V )$ and
a group $G$ acting freely on $\ti\K$ and $\ti\V$, we construct the bimodule problem $\ca=(\K ,\V )$ and a covering morphism $\py =(\py_0,\py_1):\ti {\ca}\myto{\ca}$ in a following way. Let the set $\Ob\K $ be the set $(\Ob{\ti\K })/G$ of orbits, and let the functor $\py_0$ be the natural projection $\Ob\ti\K\to\Ob\K$ on objects. For objects $\ti A,\ti B\in\Ob\ti\K$,
let us identify an element $\vi\in{\ti{\K}}(\ti A,\ti B)$ ($a\in{\ti{\V}}(\ti A,\ti B)$ respectively)
with the corresponding element of sum $\ds\oplus_{\ti X\in \py_0^{-1}(\py_0(\ti{A})),\ti Y\in \py_0^{-1}(\py_0(\ti{B}))}{\ti\K }(\ti X,\ti Y)$
($\ds\oplus_{\ti X\in \py_0^{-1}(\py_0(\ti{A})),\ti Y\in \py_0^{-1}(\py_0(\ti{B}))}{\ti\V }(\ti X,\ti Y)$).
For $A,B\in\Ob\K $ we define
$$\ds\K (A,B)=(\hskip-2.5em\dsum_{\ti X\in \py_0^{-1}(A),\ti Y\in \py_0^{-1}(B)}\hskip-2.5em{\ti\K }(\ti X,\ti Y))/\ti\K_G,\quad
\ds\V (A,B)=(\hskip-2.5em\dsum_{\ti X\in \py_0^{-1}(A),\ti Y\in \py_0^{-1}(B)}\hskip-2.5em{\ti\V }(\ti X,\ti Y))/\ti\V_G$$
where $\ti{\K}_G$ ($\ti\V _G$ respectively) is the subspace
generated by $\vi-g\vi$ ($a-ga$)
provided $\vi$ ($a$) runs ${\ti\K }(\ti X,\ti Y)$ (${\ti\V }(\ti X,\ti Y)$) for all $\ti X,\ti Y\in\Ob\ti\K$ such that
$\py_0(\ti X)=A$, $\py_0(\ti Y)=B$, and $g$ runs $G$.
We denote the class of $\vi$ ($a$) by $G\vi\in\K (A,B)$ ($Ga\in\V (A,B)$).
For $\vi\in{\ti{\K}}(\ti X,\ti Y)$, $\psi\in{\ti{\K}}(g\ti Y,\ti Z)$,
the composition of $G\vi$ and $G\psi$ is defined by $G(b(ga))$, and
the sum of $Ga$ and $Gc$ for $c\in{\ti{\K}}(g \ti X,g\ti Y)$ is $G(ga+c)$. Now we can define the functor $\py_0$ on morphisms by the map ${\ti{\K}}(\ti A,\ti B)\ni\vi\omto G\vi\in\K(\py_0(\ti A),\py_0(\ti B))$.
The $\K$-bimodule structure on $\V$ and the map $\py_1$ are defined similarly.

The bimodule problem $\ca$ is called a $G$-\define{quotient} of $\ti\ca$ and is denoted by ${\ti\ca}/G$.
The constructed morphism $\py_G=\py:{\ti\ca}\to\ca$ of bimodule problems
mapping an object $\ti X$ to $G \ti X$ and an arrow $a:\ti X\to \ti Y$ to $Ga:G\ti X\to G\ti Y$ is correctly defined,
is a covering morphism and is called a \define{quotient morphism} of the bimodule problems.

A covering isomorphic to the defined above covering
$\py_G:{\ti\ca}\myto{\ti\ca }/G$
is called \define{a Galois covering} with the \define{fundamental} group $G$.

\begin{theorem}[\cite{DOF}] \label{t_dof}
Let $\py_G:{\tica}\myto{\A}$ be a Galois covering of a bimodule problem $\A$ with a fundamental group $G$, and let the covering $\tica$ be locally representation-finite. Then the push-down functor $\rep\py_{G}:\rep\tica\myto\rep\A$ is a Galois covering of $\rep\A$ with the fundamental  group $G$, and $\A$ is locally representation-finite.
If $\A$ is finite dimensional, then $\A$ is of finite representation type.
\end{theorem}

\section{Existence of a standard minimal non-schurian bimodule subproblem}

\subsection{2-dimensional complex for one-sided bimodule problem.}
A $2$-dimensional \define{cell complex} $\bSigma=(\bSigma_0,\bSigma_1,\bSigma_2)$
associated with a bimodule problem $\ca\in\cC$ having a quasi multiplicative basic bigraph $\si=( \Sigma_0,\Sigma_1)$ (see \cite{bgoadm2011})
consists of:

\begin{orditem}
\item the sets $\bSigma_0=\siz$, $\bSigma_1=\sio$;

\item
the set $\bSigma_2$ of $2$-cells formed by the set $\rT$ of \define{triangles} $\sC=(a,b,\vi)$ for a $\vi\in\Si_1^1$ such that $\con_a(\vi b)\ne0$, and the set $\rQ$ of \define{quadrangles} $\lozenge$ for a joint arrow $\vi\in\si_1^1$ with  $\rP_{\vi}=\{(a_1,b_1),(a_2,b_2)\}$, $a_1\ne a_2$, $b_1\ne b_2$. We can depict quadrangle by a bigraph
$
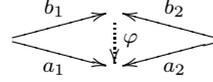
\xymatrix@C0.50cm@R4pt{
&&\ar@{..>}[dd]^{\vi}&&\\
\ar[drr]_{a_1}\ar[urr]^{b_1}  && &&\ar[dll]^{a_2}\ar[ull]_{b_2}\\
&&&&}$

\item
a map $\partial=\partial_\bsi:\bsit\myto\cycles_{\Sigma}$ is defined by $\partial(\sC)=\langle \vi ba^{-1}\rangle$ for a triangle $\sC$, and by $\partial(\lozenge)=\langle a_2^{-1} a_1 b_1^{-1} b_2\rangle$ for a quadrangle $\lozenge$.
\end{orditem}

% besides, $\rT=\rT_0\cup\rT_1$ where $\rT_i=\{\sC\in\rT \mid a,b\in\si_1^i\}$;  $\Im\partial\subset \rcycles_\bsi$.

%A  complex $\bSigma=\bSigma_\ca$, associated with an admitted bimodule problem $\ca$
%with a qm basis $\bSi=\bsi_\ca$ consists of

%parameterized by the joint arrows $\vi\in\sioo$.
%The bigraph structure $\Delta$ inherits from $\si$.

\begin{remark}
Definition of a quasi multiplicative basis for a bimodule problem $\ca\in\cC$ implies the following cell properties:
\begin{ordenumbr}

\item\label{geombas1}
every $\vi\in\sioo$ with $\rP_\vi\ne\varnothing$ belongs either to one or two triangles from the set $\rT_0=\{(a,b,\vi)\in\rT \mid a,b\in\si_1^0\}$;

\item\label{geombas3}
for $a,b\in\si_1^0$, $a\ne b$, such that $s(a)=s(b)$, the pair $(a,b)$
belongs to at most $2$ triangles from $\rT_0$;

\item\label{geombas2}
if $(a_1,b_1,\vi)$, $(a_2,b_2,\vi)\in\rT_0$ are different triangles, then $a_1\ne a_2$, $b_1\ne b_2$;

\item\label{geombas4}
if $(a,b,\vi_1), (a,b,\vi_2)\in\rT_0$ and $\vi_1\ne \vi_2$,
then there exist $(a_i,b_i,\vi_i)\in\rT_0$, $i=1,2$, such that  $a_1\ne a_2$, $b_1\ne b_2$.
\end{ordenumbr}
\end{remark}

\subsection{Universal covering associated
with a quasi multiplicative basis of schurian bimodule problem.}

%Quasi multiplicative basis $\si$ of schurian bimodule problem $\ca$ allows to construct a Galois covering of $\ca$ in the following way.

\begin{lemma}\label{Lemma-le_concov}
Let $\ca\in\cC$ be a connected locally finite dimensional bimodule problem, let $\bsi=\bsi_\ca$  be 2-dimensional cell complex associated with a quasi multiplicative basis $\si=\si_\A$ of $\ca$,
and let $\py_{\bsi}:\btisi\to\bsi$ be a covering of abstract complexes.
Then there exist a locally finite dimensional bimodule problem $\ti\ca=({\ti\K },{\ti \V })\in\cC$ with a basis $\tisi$ and associated 2-dimensional cell complex $\btisi$, and a Galois covering morphism $\py_\ca:\ti\ca\to\ca$ with the fundamental group $\sg(\bsi)$ such that the diagram\vspace*{-2ex}
$$
\begin{CD}\label{comdiabas}
 \tisi  @>{i_\tisi} >>\ti\ca  \\[-1ex]
 @V\py_\bsi VV      @VV\py_{\ca} V\\[-1ex]
 \si @>i_\si>> \ca\\
\end{CD} $$\vskip-1ex\noindent
commutes, where $i_\si: \si \to \ca$ and $i_\tisi: \tisi \to \ca$ are natural embeddings.
\end{lemma}

%$$ i_\bsi: \bsi \to \ca \qquad\text{\todo{What is definition?}}$$

\begin{proof}
Let $\Ob{\ti \K  }=\btisi_0$, and let the spaces $\Rad{\ti\K }(\ti X,\ti Y)$,
${\ti \V } (\ti X,\ti Y)$ be freely generated over $\k$ by $\btisi_1^1(\ti X,\ti Y)$
and $\btisi_1^0 (\ti X,\ti Y)$ correspondingly, $\ti X,\ti Y\in\Ob{\ti\K }$.
If ${\ti a}:\ti X\to \ti Y$, ${\ti b}:\ti Y\to \ti Z$ are two elements of $\btisi_1$,
$\py_1({\ti a})=a$, $\py_1({\ti  b})=b$, then for any $x\in\bsi_1$ such that $\con_x (ba)\ne0$ and for a unique $\ti x\in\btisi_1$ such that $\py_1({\ti x})=x$, $s({\ti x})=s({\ti a})$, we set
$\con_{\ti x}({\ti b}{\ti a})=\con_x (ba)$ in $\ti\A$, and $\con_{\ti y}({\ti b}{\ti a})=0$ for any other $\ti y\in\btisi_1$.
The composition ${\ti b}{\ti a}$ is correctly defined since $\langle{\ti b}{\ti a}{\ti x}^{-1}\rangle$ is a bound of a cell in $\btisi$, and
hence it is a cycle in $\btisi$.
Associativity of such composition is obvious.
$\ti{\K} $-bimodule structure on ${\ti \V }$ is defined similarly.
The covering morphism $\py_\ca:\ti\ca\myto\ca$ is uniquely defined by the commutativity of the diagram.
\end{proof}

%of the diagram \ref{comdiabas}.
%\footnote{diagram \ref{comdiabas} \quad ?????}

\begin{remark}\label{lifting}
Let $\py_\ca:\ti\ca\to\ca$ be a Galois covering morphism of bimodule problems, and let $\si$ be a basis of $\ca$.
Then inverse images of elements of $\si_{1}$ form a basis $\ti\si$ of $\ti\ca$, and $\py_\ca$
induces the associated covering $\py_\bsi:\ti\bsi\to \bsi$ of complexes.
This construction is inverse to one in Lemma \ref{Lemma-le_concov}.
\end{remark}

For a bimodule problem $\A\in\cC$, let $\py:\btisi\to\bsi$ be the constructed above universal covering of 2-dimensional complex $\bsi=\bsi_\ca$. By Lemma~\ref{Lemma-le_defshur}, there exists the corresponding to $\py$ covering morphism $\py_\ca:{\ti\ca}\to\ca$ of bimodule problems which we call
\define{an universal covering of bimodule problem  $\ca$ associated with the complex $\bsi$.}

A bimodule problem $\ca$ is said to be \define{(geometrically) simply connected}
provided 2-dimensional complex $\bsi$ over $\si$ is connected and its fundamental group $\sg(\bsi)$ is trivial.
Obviously, simply connected schurian bimodule problem $\ca$ is isomorphic to its universal covering $\ti \ca$.

%\begin{lemma}\label{Lemma-le_anns}
%If \bp $\ca$ is non-faithful, then every sincere $M\in \ind\ca$ is non-schurian.
%\end{lemma}
%
%\begin{proof}
%Let $\Ann_\K (\V )(A,B)\ne 0$. We assume, $A\ne B$, the case $A= B$ is analogous.
%Then the subspace generated by $\Ann_\K({\V})(A,B)$
%in $\add \K(M_\K,M_\K)$ consists of nilpotent endomorphisms of $M$, hence $M$ is non-schurian.
%%$\left(\begin{array}{ccc}
%%0&0 &0  \cr 0&0 &0  \cr 0&
%%{\Ann_\K({\V})}(A,B)   &0 \cr
%%\end{array}\right)$
%\end{proof}
%
%Therefore, each sincere schurian \bp is faithful.

\subsection{Minimal non-schurian bimodule problem.}

Recall that for a sincere schurian bimodule problem $\ca$, its basis $\si_\ca$ is $0$--connected.
If \bp $\ca$ is non-faithful, then every sincere $M\in \ind\ca$ is non-schurian.
Therefore, each sincere schurian \bp is faithful.

Let $\ca=(\K,\V)$ be a  sincere non-schurian bimodule problem such that Tits form $q_\ca\in \WPS$,
and for every proper subset $S\subset \Ob\K$ the restricted bimodule problem $\ca_S$ is schurian.
Then  $\ca$ is called a \define{minimal non-schurian bimodule problem}.

Let $\cb,\cc$ be bimodule problems defined by their bigraphs:
$(\si_\cb)_0=\{X\}$, $(\si_\cb)_1=\varnothing$, $(\si_\cc)_0=\{X,Y\}$, $(\si_\cc)_1=(\si_\cc)_1^0=\{a:X\to Y\}$.
If $\ca\in\cC$, $|(\si_\A)_0|\leqs 2$, then $\ca$ is sincere schurian if and only if $\ca\in\{\cb,\cc\}$.
%$\ca=\cb$ or $\ca=\cc$.

This observation excludes the non-schurian problems having at most two vertices from the consideration,
and helps to describe the minimal non-schurian bimodule problems containing at least $3$ vertices.

\begin{lemma}[\cite{BGOR,bg2}]\label{Lemma-minnonshstr}
Let $\A=(\K,\V )$ be a  minimal non-schurian admitted bi\-module problem with a basis $\si$, and $|\siz|\geqs 3$. If $\A_{\red}=(\K/\!\Ann_\K\bsV,\bsV)$ is a sincere schu\-ri\-an \bp, then there exist two uniquely defined vertices $A,B\in\si_0^+$ such that:
\begin{ordenumbr}
\item
for any sincere $M\in \ind\A$, the vector $\dim_\A M$ is a basic root of Tits quadratic form $q_{\A_{\red}}$
with the singular vertices $A,B$;

\item
if $\Ann_\K \V (A_1,B_1)\ne0$, then the sets $\{A_1,B_1\}$, $\{A,B\}$ coincide.
\end{ordenumbr}
\end{lemma}

A minimal non-schurian bimodule problem $\ca$ satisfying the conditions of Lemma \ref{Lemma-minnonshstr}
is called a \define{standard minimal non-schurian bimodule problem with singular vertices $A$ and $B$}.

\subsection{Schurity and coverings: the main result.}

%\section{Covering morphisms and schurian bimodule problem}

%The Theorem below prove, that either Tits form $q_{\ti\ca}\not\in \WPS$ or $\ti\ca$ is a schurian bimodule problem.
%This statement below give an explanation to the notion of a standard minimal non-schurian bimodule problem
%from the point of view of covering theory (in fact it can be applied to more general class of matrix problem).
%We prove that the minimal non-schurian admitted bimodule problem, such that  $|\siz|>2$, is standard.}

%%%{Non-schurian bimodule problem with schurian covering contains a standard minimal non-schurian bimodule problem}

By the construction of universal covering, we assume that $\py_i(\tisi_i)=\si_i$, $i=0,1$, for the bases $\tisi$ and $\si$ of $\ti\ca$ and $\ca$ respectively.

\begin{theorem}
\label{theorem-support-indecomposable-and-covering}
Let $\ca\in\cC$ be a connected finite dimensional bimodule problem with weakly positive Tits form $q_\ca$, let $\ti\ca\in\cC$ be a schurian bimodule problem, and let $\py: \ti\ca{\xrightarrow{\hspace{10pt}}} \ca$ be an universal covering.
Then either $\ca$ is schurian, or contains a dotted loop, or some restriction $\ca_S$ is a standard minimal non-schurian bimodule problem.
%(see definition \ref{def_standard minimal non-schurian}).
\end{theorem}

\begin{proof}
Suppose that $\ca=(\K,\V)$ is not schurian and does not contain dotted loops. Let $\si$ and
$\ti\si$ be the  bases of $\ca$ and $\ti\ca$ correspondingly
such that $\py(\ti\si)=\si$.
We will mark the object (vertex, arrow, etc.) related to the covering $\ti\ca$
by the sign $\ti{\ }$, and its image in $\ca$ will be denoted by the same letter without $\ti{\
}$.

%% {($\bA$)}\
%We can assume, that if $S\subset\siz$ and $|S|\leqslant2$, then the bimodule problem $\ca|_S$ is schurian.

%{$\bA$}
%Indeed, if $|\ti S|=1$ holds, then the bigraph $\si$ of $\ca$ contains a dotted loop, that contradicts {$\bA$}.
%Let $\ti S=\{A_1,A_2\}$.
%Since $q_\ca\in \WPS$, then $A_1, A_2$ should connected by solid and by dotted arrows as well.
%Then the representation $X\in\ind\ti\ca$, such that $\dim_{\ti\ca} X=e_{A_1}+ e_{A_2}$ of general
%position in its dimension  is non-schurian, that also contradicts {$\bA$}.

Since bimodule problem $\ti\ca$ is schurian and $\ca$ is not, there exists a representation $\ti X \in\ind \ti\ca$ such
that for $\ti S =\supp \ti X$, the induced morphism  $\py_{\ti S}:$
$\ti\ca_{\ti S}\myto\ca_S$ is not an isomorphism. Then the restriction $\ca_S$ is sincere minimal non-schurian.
Let us choose $\ti X$ with minimal possible $|\ti S|$.

Remark that for any $S\subset\siz$, $|S|\leqslant2$, the restriction $\ca_S$ is schurian by definition of class $\cC$ and our assumption. So we can assume $|S|\geqs3$, and hence $|\ti S|\geqs3$. By Remark \ref{lifting}, the induced by $\py$ morphism $\py_{\ti S}:\btisi_{\ti S}\myto\bsi_S$ of the associated complexes is not an isomorphism. Then:

1) either the induced by $\py$ map
$\py_0|_{\ti S}: (\tisi_{\ti S})_{0}\to (\si_S)_{0}$ is not a bijection on the vertices;

2) or the map $\py_0|_{\ti S}$ is a bijection,
but there exist ${\tial}_1,{\tial}_2\in\ti S$ such that
$\py(\tisi_1({\tial}_1, {\tial}_2))\ne\si_1(A_1,A_2).$

\setcounter{step}{0} \best{}
The first case is impossible.
\enst

%\begin{proof}
\define{Proof.}
Suppose, there exist  $\tial_1,\tial_2\in\ti S$, $\tial_1\ne\tial_2$, such that  $A_1=A_2\in S$. Since $\ti\A_{\ti S}$ is schurian, by Lemma~\ref{Lemma-le_defshur}, $\ti x=\dim\ti X$ is a sincere positive root of weakly positive unit quadratic form $q_{\ti\A_{\ti S}}$. If $\ti x$ is not basic with singular vertices $\tial_1$, $\tial_2$, then, by \cite[Lemma~1]{bg2}, there is a non-sincere $\ti y<\ti x$ such that $\tial_1,\tial_2\in\supp\ti y$ which contradicts to the minimality of $|\ti S|$ by Lemma~\ref{Lemma-le_defshur}. Therefore, $\ti x$ is a basic root with the singular vertices $\tial_1$, $\tial_2$. By definition of basic root, the vertices $\tial_1$ and $\tial_2$ are defined uniquely, i.~e. $B_1\ne B_2$ for any pair $\tibe_1,\tibe_2\in\ti S$ of different vertices such that $\{\tibe_1,\tibe_2\}\ne\{\tial_1,\tial_2\}$.

Since $\ti\A_{\ti S}$ is $0$-connected, there is a vertex $\tiga\in\ti S$ connected with $\tial_1$ by a solid arrow
${\ti{a}}:\tiga\myto\tial_1$ (up to direction of $\ti{a}$). If $\tiga=\tial_2$, then the quadratic form $q_{\ti\A_{\ti S}}$ is not \WP. So $\tiga\ne\tial_2$. By Lemma \ref{Lemma-le_defshur}, for a positive root $\ti z=\ti x - e_{\ti A_1}$, there is an indecomposable representation $\ti Z$ of the dimension $\ti z$ with $|\supp \ti Z|<|\ti S|$. The corresponding restriction $\py_{\supp \ti Z}:\ti\A_{\supp \ti Z}\to \A_S$ is not an isomorphism since the arrow $a: C\to A_1$ does not have an inverse image in $\ti\si_1(\ti C, \ti A_2)$, which contradicts to the minimality of $|\ti S|$.
\hfill $\qed$

Now we can assume that the restriction $\py_0|_{\ti S}$ is a bijection, and the case 2) holds.

\best{}
Let $\tial_1,\tial_2\in\ti S$ and $\py(\tisi_1(\tial_1,\tial_2))\ne\si_1(A_1,A_2)$.
Then the set of vertices $\{\tial_1,\tial_2\}$ is uniquely defined, and $\dim_{\ti\A}\ti X$ is a basic root of $q_{\ti\A_{\ti S}}$ with singular vertices $\tial_1,\tial_2$.
\enst

The proof is similar.

\best{}
The set $$\Chi=\big(\si_1(A_1,A_2)\setminus\py_1(\tisi_1(\tial_1,\tial_2)\big)
\cup  \big(\si_1(A_2,A_1) \setminus \py_1(\tisi_1(\tial_2,\tial_1)\big)$$ consists of dotted arrows.
\enst

%\begin{proof}
\define{Proof.}
Since bimodule problem $\A$ is admitted, either $\Chi\subset\si_1^0$, or $\Chi\subset\si_1^1$. Suppose that $\Chi\subset\si_1^0$.
Let $x=\dim_\ca\rep\py(\ti X)$. Then $x_A={\ti x}_\tial$ for any $\tial\in\ti S$. Hence
\begin{gather*}
\ds{ q_\ca(x)={\sum}_{(A,B)\in S\times S } (\delta_{AB}+|\sioo(A,B) |- |\sizo(A,B) |) x_A x_B= }
\\
={{\sum}_{(\tial,\tibe)\in \ti S } (\delta_{\tial\tibe}+|\tisi_1^0(\tial,\tibe)|
- |\tisi_1^0(\tial,\tibe) |) x_\tial x_\tibe - } {\sum}_{a\in\Chi} x_{A_1} x_{A_2}=
\\
=\ds{ q_{\ti\ca}(\tix)-{\sum}_{a\in\Chi} x_{A_1} x_{A_2}= } 1-\ds{\sum}_{a\in\Chi} x_{A_1} x_{A_2} \leqslant 1-1=0,
\end{gather*}
that contradicts to the weak positivity of $q_\ca$.
%\end{proof}
\hfill $\qed$

\best{}
$\Chi\subset\Ann_{\K _S}\V _S$.
\enst

\define{Proof.}
If $\vi\in\Chi\backslash\Ann_{\K _S}\V _S$, then (up to direction of $\vi$) there is a triangle $(a_1, a_2,\vi)\in \rT_0$ with
$a_1: B\myto A_1$, $a_2: B\myto A_2$ for some $B\in S$.
Since the triangles lifts up in the covering, then for the unique $\tibe\in\ti S$ there exist
$\tia_1:\tibe\myto\tial_1,\tia_2:\tibe\myto\tial_2$ in $\tisi$ and a
triangle $(\tia_1,\tia_2,\ti\vi)\in{\ti \rT}_0$.
Therefore $\py_0(\ti \vi)=\vi$ that contradicts to the definition of $\Chi$.
%\hfill $\qed$
\end{proof}

Hence, $(\A_S)_{\red}=(\K _S/\Ann_{\K _S}\V _S,\V _S)$ is isomorphic to $\ti\A_{\ti S}$, and therefore $(\A_S)_{\red}$ is sincere schurian. By Lemma~\ref{Lemma-minnonshstr}, $\A_S$ is standard minimal non-schurian bimodule problem, which completes the proof of Theorem.
%\begin{remark}
%The Theorem prove, that either Tits form $q_{\ti\ca}\not\in \WPS$ or $\ca$ is a schurian bimodule problem.
%\end{remark}

%%%
%%%\subsection{Support of an indecomposable and
%%%coverings}
%%%

%From the last step of the proof of Theorem \ref{theorem-support-indecomposable-and-covering} follows
%
%\begin{corollary}
%Let $\py:$
%$\ti\ca\myto\ca$ be a covering of bimodule problems,
%such that  $\ti\ca$ is schurian and the bimodule problem $\cb=$
%$\ca|_{S}$ for some set of objects $S$ is
%sincere.
%
%\begin{enumerate} \item  If $\cb$  is standard non-schurian,
%then the $0$- connected components of
%$\ti\ca|_{\py_0^{-1}(S)}$ are isomorphic to
%$\cb_{\red}$ and the composition of the canonical morphisms
%$$\ti\ca|_{\ti S}\xar{\ }
%\cb\xar{\ }\cb_{\red}$$ is an isomorphism.
%
%\item   If $\cb$ schurian, then
%$\ti\ca|_{\py_0^{-1}(S)}$ is a direct sum of
%copies of $\cb$.
%\end{enumerate}
%\end{corollary}

%\end{document}

\section*{Conclusion}

The article is a part of research of the representation finiteness problem
for a wide class of multi-vector space categories consisting of so called one-sided bimodule problems.
We use the construction of the universal covering of an admitted bimodule problem
in order to obtain some schurity criterium for the bimodule problems from our class.
We are going to study representation type of one-sided bimodule problems using developed technique.

\bibliographystyle{pigc_plain}
\bibliography{biblio.bib}

%===============================================================================
%   Information about authors
%===============================================================================

\printArticleAuthorsInfo{\thearticlesnum}

\end{document}